\documentclass[10pt]{amsart}

\usepackage{enumerate}
\usepackage{color}

\newtheorem{theorem}{Theorem}[section]
\newtheorem{proposition}[theorem]{Proposition}
\newtheorem{lemma}[theorem]{Lemma}

\newtheorem{definition}[theorem]{Definition}
\newtheorem{remark}[theorem]{Remark}

\newtheorem{assumption}[theorem]{Assumption}

\newcommand{\tlangle}{\text{\ensuremath{\langle\hspace*{-0.5ex}\langle}}}
\newcommand{\trangle}{\text{\ensuremath{\rangle\hspace*{-0.5ex}\rangle}}}

\begin{document}

\title[Integration and differential equations for typical paths]{Stochastic integration and differential equations for typical paths}
\author{Daniel Bartl$^\ast$
\and Michael Kupper$^\times$
\and Ariel Neufeld$^+$} 
\thanks{
	$^\ast$Department of Mathematics, University of Vienna, daniel.bartl@univie.ac.at. \\
  $\phantom{ddd}$$^\times$Department of Mathematics, University of Konstanz, kupper@uni-konstanz.de.\\ 
  $\phantom{ddd}$$^+$Division of Mathematical Sciences, NTU Singapore,
  ariel.neufeld@ntu.edu.sg.
}
\keywords{F\"ollmer integration; Pathwise Stochastic Intergral; Pathwise SDE;\\ 
	$\phantom{ddd}$Infinite
	dimensional stochastic calculus; Vovk's outer measure} 
\date{\today} 
\subjclass[2010]{60H05, 
60H10; 
60H15; 
 91G20  
}

\begin{abstract} 
The goal of this paper is to define stochastic integrals and to solve stochastic differential equations for typical paths taking values in a possibly infinite dimensional separable Hilbert space without imposing any probabilistic structure.
 In the spirit of \cite{PerkowskiPromel16,Vovk12} and motivated by the pricing duality result obtained in \cite{bartl2017pathwise} we introduce an outer measure 
as a variant of the pathwise minimal superhedging price where agents are allowed to trade not only in $\omega$ but also in $\int\omega\,d\omega:=\omega^2 -\langle \omega \rangle$ and where they are allowed to include beliefs in future paths of the price process expressed by a prediction set. We then call a property to hold true on typical paths if the set of paths where the property fails is null with respect to our outer measure.
 It turns out that adding the second term $\omega^2 -\langle \omega \rangle$ in the definition of the outer measure enables to directly construct stochastic integrals which are continuous, even for typical paths taking values in an infinite dimensional separable Hilbert space. Moreover, when restricting to continuous paths whose quadratic variation is absolutely continuous with uniformly bounded derivative, a second construction of  model-free stochastic integrals for typical paths is presented, which then allows to solve in a model-free way stochastic differential equations for typical paths.
\end{abstract}

\maketitle
\setcounter{equation}{0}

\section{Introduction}
{In this paper we investigate the problem of 
constructing pathwise stochastic integrals 
as well as solutions of
 stochastic differential equations without a reference probability measure.
 It is well-known that defining a stochastic integral is a highly non-trivial problem and cannot be deduced directly from classical measure-theoretical calculus, as in general, stochastic processes describing the noise of the dynamics do not have finite variation paths. }
The It\^o integral and the corresponding It\^o calculus have been developed  overcoming the obstacle of how to define integrals and differential equations when noise occurs. However, its construction heavily depends on a probabilistic structure and cannot be defined pathwise. More precisely, the construction of the stochastic integral is accomplished by a $L^2(P)$-limit procedure, and from the Bichteler-Dellacherie theorem it is known that the only class of good integrators for which the integral is, in a suitable sense, a continuous operator, are semimartingales. 

Later, there were several approaches to define stochastic integrals pathwise, without assuming any probabilistic structure. This allows to consider more general paths as integrators, rather than semimartingale paths. Moreover, motivated from mathematical finance, pathwise stochastic calculus can be employed to price financial derivatives without assuming any probabilistic model on the financial market leading to robust prices; see \cite{AcciaioBeiglbockPenknerSchachermayer.16,bartl2017pathwise,bartl2017duality,beiglbock2017pathwise,BurzoniFrittelliMaggis.16,DolinskySoner.12} to name but a few.
{The first result which provides a construction of a stochastic integral without imposing any probabilistic structure was given in F\"ollmer \cite{Follmer81}.
In Bichteler~\cite{Bichteler81} and Karandikar~\cite{Karandikar95}  a pathwise construction of the stochastic integral was proposed for c\`adl\`ag integrands which enables to solve the so-called aggregation problem of defining  a stochastic integral which coincides with the classical stochastic integral simultaneously for all semimartingale measures. This has been used to price financial derivatives under Knightian uncertainty, see \cite{NeufeldNutz13,NutzSoner12,SonerTouziZhang13}. A solution for the above aggregation problem, under the continuum hypothesis, was obtained in Nutz~\cite{Nutz12}  for general predictable integrands using medial limits. 
In Lyons~\cite{Lyons95} F\"ollmer's pathwise stochastic calculus has been extended to obtain prices of American and European options under volatility uncertainty. In Davis--Obloj--Raval~\cite{DavisOblojRaval14} F\"ollmer's pathwise stochastic calculus has been employed to price weighted variance swaps when a finite number of European call and put options for a known price are traded.
%
%
%
%
%
In  Cont--Fourni\'e~\cite{ContFournie10} pathwise stochastic integrals with a directional derivative (whose construction goes back to Dupire~\cite{Dupire09}) of a non-anticipative functional as integrand have been constructed and a change of variable formula for such integrals was obtained. Using this framework, an It\^o isometry for such integrals was established in Ananova--Cont~\cite{AnanovaCont16}, whereas in Riga~\cite{Riga16} a pathwise notion for the gain process with respect to corresponding self-financing trading strategies was introduced. 
In Cont--Perkowski~\cite{cont2018pathwise} it was shown that one can extend F\"ollmer's pathwise It\^o calculus to paths
with arbitrary regularity by employing the concept of $p$-th variation
along a sequence of time partitions.
Moreover, during the reviewing process of
this work, the existence and uniqueness of pathwise SDEs has been established in Galane--Lochowski--Mhlanga~\cite{galane2018sdes} using the similar approach of deriving a model-free version of the Burkholder--Davis--Gundy inequality.
Furthermore, for pathwise construction of stochastic integrals with respect to c\`adl\`ag integrators, we refer to Hirai~\cite{Hirai17} as well as Lochowski--Perkowski--Pr\"omel~ \cite{LochowskiPerkowskiPromel17}.}

Recently, motivated by game-theoretic considerations, Vovk introduced in \cite{Vovk12} an outer measure on the space of continuous paths and declared an event to be typical if its complement is null with respect to the defined outer measure. He then showed that typical paths possess a quadratic variation. In other words, it was shown in \cite{Vovk12} that paths which do not possess a quadratic variation allow a form of arbitrage.
 Vovk's approach was employed in Perkowski--Pr\"omel~\cite{PerkowskiPromel16} to define an outer measure which can be interpreted as the pathwise minimal superhedging price motivated from financial mathematics.  Using their outer measure they constructed a model-free stochastic integral which is continuous for typical price paths and connected their typical paths with rough paths by demonstrating that every typical price path possess an It\^o rough path. Moreover, Vovk~\cite{Vovk15purely} and Vovk--Shafer~\cite{VovkShafer17} provide several additional constructions of model-free stochastic integrals for typical paths. In \cite{peng2007g,peng2010nonlinear} It\^o calculus with respect to the so-called $G$-Brownian motion as integrator has been developed by Peng, which is motivated from financial mathematics when investigating pricing and portfolio optimization problems under volatility uncertainty. There, by referring to the notion of typical paths, the pathwise integral and the corresponding stochastic calculus is defined for typical paths with respect to the so-called $G$-expectation. 
 
{The goal of this work is to provide a construction of a model-free stochastic integral for typical paths which allows to solve stochastic differential equations pathwise. 
	Our setting is similar to the one in Perkowski--Pr\"omel~\cite{PerkowskiPromel16}. More precisely, we introduce an outer measure which is defined as a variant of the pathwise minimal superhedging price and call a property to hold true on typical paths if the set of paths where the property fails is null with respect to our outer measure. The main difference, compared to the outer measure in \cite{PerkowskiPromel16}, is that in our definition hedging is not only allowed in $\omega$ representing the price path of the risky security, but also  in the second security $\omega^2-\langle \omega \rangle$. This roughly means that superhedging strategies both in $\omega$ and $\int \omega \,d\omega$ are permitted.
It turns out that adding the second term in the definition of the outer measure enables to directly define stochastic integrals which are continuous, even for paths taking values in an infinite dimensional separable Hilbert space, see Theorem~\ref{thm:integral}. Its proof is based on an elementary, but crucial observation provided in Lemma~\ref{lem:integrals.squared} using heavily the second order term in the definition of the outer measure, which is then employed to derive a Burkholder--Davis--Gundy (BDG) type of inequality, see Proposition~\ref{prop:bdg.calE}.  
To be able to solve stochastic differential equations pathwise, a second construction of model-free stochastic integrals is provided 
when restricting to all paths possessing an absolutely continuous quadratic variation whose derivative is uniformly bounded, see Theorem~\ref{thm:integral.Xi.has.assumptions}. This notion of a model-free stochastic integral allows us to solve stochastic differential equations for typical paths taking values in a possibly infinite dimensional Hilbert space, see Theorem~\ref{thm:sde}.}

The remainder of this paper is organized as follows. In Section~\ref{sec:SetupMainResults}, we introduce the setup
and state our main results, 
 whose proofs are then provided in Section~ \ref{sec:Proof}.

\section{Setup and main results}
\label{sec:SetupMainResults}
Let $H$ be a separable Hilbert space with scalar product $\langle\cdot,\cdot\rangle_H$ and respective norm $\|h\|_H=\langle h, h\rangle_H^{1/2}$. For a finite time horizon $T>0$ we denote by $C([0,T],H)$ the space of all continuous paths $\omega\colon [0,T]\to H$ endowed with the supremum norm $\sup_{t\in[0,T]} \|\omega(t)\|_H$. Let $\Omega$ be the Borel set of all $\omega\in C([0,T],H)$ 
for which the \emph{pathwise quadratic variation} $\langle\omega\rangle$ given by
\[\langle\omega\rangle_t:= \lim_{m\to\infty} \sum_{k=1}^\infty 
\Big(\|\omega\big(\sigma^m_{k}(\omega)\wedge t\big)\|_H 
- \|\omega\big(\sigma^m_{k-1}(\omega)\wedge t\big)\|_H\Big)^2\]
exists as a limit in the supremum norm in $C([0,T],\mathbb{R})$ along the dyadic partition
\[\sigma^m_k(\omega):=\inf\left\{t\geq \sigma^m_{k-1}(\omega) : \|\omega(t)-\omega(\sigma^m_{k-1}(\omega))\|_H\ge 2^{-m}\right\}\]
with $\sigma^m_0(\omega)=0$ for all $k,m\in\mathbb{N}$. 
Define the processes $S\colon[0,T]\times\Omega\to H$ and $\mathbb{S}\colon[0,T]\times\Omega\to \mathbb{R}$ by
\[ S_t(\omega):=\omega(t)\qquad\text{ and }\qquad
\mathbb{S}_t(\omega):=\|S_t(\omega)\|_H^2 -\langle\omega\rangle_t,\]
and let $\mathbb{F}$ be the raw filtration on $\Omega$ given by $\mathcal{F}_t:=\sigma(S_s: s\leq t)$ for all $t\in[0,T]$ and $\mathbb{F}_+$ its right-continuous version. 

Given another separable Hilbert space $K$, 	denote by $L(H,K)$ the Banach space of all
bounded linear operators $F\colon H \to K$ endowed with the operator norm $\|F\|_{L(H,K)}:=\sup_{\|h\|_H\le 1}\|F(h)\|_K$. We denote by $\mathcal{H}_s(H,K)$ the set of all \emph{simple integrands}, i.e.~processes 
$F\colon [0,T]\times \Omega\to L(H,K)$ of the form 
\[F_t(\omega)=\sum_{n=0}^\infty f_n(\omega)1_{( \tau_n({\omega}),\tau_{n+1}({\omega})]}(t)\]
where $0=\tau_0\leq\dots \leq \tau_n\leq\tau_{n+1}\leq \dots\leq T$ are $\mathbb{F}_+$-stopping times such that for each $\omega$ there is $n(\omega)$ such that $\tau_{n(\omega)}(\omega)=T$ and the functions
$f_n\colon\Omega\to L(H,K)$ are $\mathcal{F}_{\tau_n+}$-measurable.
For such a simple integrand $F$ the stochastic integral 
$(F\cdot N)$ against any process $N\colon[0,T]\times\Omega\to H$
can be defined pathwise
\[ (F \cdot N)_t(\omega)
:=\sum_{n=0}^\infty f_n(\omega)\big(N^t_{\tau_{n+1}}(\omega)-N^t_{\tau_n}(\omega)\big)
\in K \]
where $N^t_s:=N_{s\wedge t}$. Notice that processes in
$\mathcal{H}_s(H,\mathbb{R})$ can be viewed as simple integrands with values in $H$ 
by identifying $L(H,\mathbb{R})$ with $H$. 


Our results strongly rely on the following modified version of Vovk's \cite{Vovk12} and Perkowski--Pr\"omel's \cite{PerkowskiPromel16} outer measure.
If not explicitly stated otherwise, 
all (in-)equalities between functions $X:\Omega\to[-\infty,+\infty]$ are understood pointwise on $\Omega$.

\begin{definition}
	\label{def:second-vovk}
	Let $\Xi\subseteq\Omega$ be a fixed prediction set. Then 
	for all $X\colon\Omega\to[0,+\infty]$ we define 
	\[ \mathcal{E}(X):=
	\inf\!\left\{ \lambda\geq 0 \,: 
	\begin{array}{l}
	\text{there are $(F^n)$ in $\mathcal{H}_s(H,\mathbb{R})$ and $(G^n)$ in $\mathcal{H}_s(\mathbb{R},\mathbb{R})$ such that}\\ 
	\text{$\lambda+(F^n\cdot S)_t+(G^n\cdot \mathbb{S})_t\geq 0$ on $\Xi$ for all $n$ and $t\in[0,T]$, and} \\
	\text{$\lambda+ \liminf_n\big( (F^n\cdot S)_T +(G^n\cdot \mathbb{S})_T \big)\geq X$ on $\Xi$ }
	\end{array}\!\right\}.
	\] 
	Moreover, we say that a property holds for \emph{typical paths} (on $\Xi$) if $\mathcal{E}(1_N)=0$ for 
	the set $N$ where the property fails. 
\end{definition}
{\begin{remark}\label{rem:finance}
	 Motivated by the work of Vovk \cite{Vovk12} and Perkowski--Pr\"omel~\cite{PerkowskiPromel16}, our outer measure  is defined as a variant of the pathwise minimal superhedging price. Here, the pathwise superhedging property only needs to hold with respect to a predefined prediction set of paths. Such a superhedging price, which can be seen as a second-order Vovk approach, was introduced in \cite{bartl2017pathwise} and enabled to provide  a pricing duality result when the financial agent is allowed to include beliefs in future
	paths of the price process expressed by a prediction set $\Xi$, while eliminating all those which are
	seen as impossible. This reduces the (robust) superhedging price, which typically leads to too high prices, see \cite{DolinskyNeufeld.16,Neufeld.17}. We refer to \cite{bartl2017duality,HouObloj.15,Mykland.03} for related works regarding prediction sets and its relation to pricing of financial derivatives.
\end{remark}}
{\begin{remark}\label{rem:Vovk-quad}
	We restrict ourselves to paths for which the quadratic variation exists to a priori be able to define our outer measure $\mathcal{E}(\cdot)$. 
	However, Vovk showed in \cite{Vovk12} that typical paths with respect to his outer measure automatically possess a quadratic variation. Comparing our outer measure with the one in \cite{Vovk12}, we can argue as in Perkowski--Pr\"omel~\cite[Lemma~2.9]{PerkowskiPromel16} that 
	our outer measure enforces our typical paths to possess a quadratic variation, meaning that in fact it was not a restriction to consider only paths with finite quadratic variation.
\end{remark}}
From now on we fix a prediction set $\Xi \subseteq \Omega$ and consider the outer measure $\mathcal{E}(\cdot)$ with respect to  $\Xi$. Further, we denote by $\mathcal{M}(\Xi)$ the set of \emph{martingale measures supported on $\Xi$}, i.e.~all Borel probability measures $P$ on $\Omega$ such that
$(S_t)$ is a $P$-$\mathbb{F}$-martingale and $P(\Xi)=1$.

The function $t\mapsto \langle\omega\rangle_t$ is continuous and nondecreasing for all $\omega\in\Omega$, thus induces a finite measure on $[0,T]$. 
Therefore, we denote 
\[( F \cdot \langle S\rangle)_t(\omega):= \int_0^t F_u(\omega)\, d\langle S(\omega)\rangle_u\]
the Lebesgue-Stieltjes integral  with respect to a function $F\colon[0,T]\times\Omega \to \mathbb{R}$ 
such that $F(\omega)$ is measurable and $\int_0^T|F_u(\omega)|\, d \langle S (\omega)\rangle_u < +\infty$ for all $\omega \in \Omega$ and set $( F \cdot \langle S\rangle)_t(\omega):=+\infty$ otherwise.

Now, we start with our first result stating that for any prediction set $\Xi\subseteq \Omega$ we can define for typical paths stochastic integrals which are continuous. To that end, for any $F\colon [0,T]\times \Omega \to L(H,K)$  we introduce
\begin{equation*}
\|F\|_{\mathcal{H}^\infty(H,K)}:=\sup_{\omega\in\Xi} (\|F\|^2_{L(H,K)} \cdot \langle S\rangle\big)_T^\frac{1}{2}(\omega)
\end{equation*}
and define the space of integrands
\[ \mathcal{H}^\infty(H,K)
:=\left\{ F\colon\Omega\times[0,T]\to L(H,K) : 
\begin{array}{l}
\|F\|_{\mathcal{H}^\infty(H,K)}<+\infty, \text{ and there exists}\\
\text{a sequence } (F^n) \text{ in } \mathcal{H}_s(H,K)\\
\text{such that } \|F-F^n\|_{\mathcal{H}^\infty(H,K)} \to 0
\end{array}\!\right\}.\]

\begin{theorem}
	\label{thm:integral}	
	Let $F \in \mathcal{H}^\infty(H,K)$ and assume that $K$ is finite dimensional. 
	Then the stochastic integral 
	\[(F\cdot S)\colon \Omega\to C([0,T],K) \]
	exists and satisfies the following weak Burkholder--Davis--Gundy (BDG) type of inequality
	\[\mathcal{E}\Big( \sup_{t\in[0,T]} \|(F\cdot S)_t\|_K^2 \Big)
	\leq 4 \sup_{\omega\in \Xi} (\|F\|_{L(H,K)}^2\cdot \langle S\rangle)_T(\omega)
	=4 \|F\|_{\mathcal{H}^\infty(H,K)}^2.\]
	Moreover, the space $\mathcal{H}^\infty(H,K)$ and the stochastic integral are linear (for typical paths)
	and the latter coincides with the classical It\^o-integral under every martingale measure
	$P\in\mathcal{M}(\Xi)$.
\end{theorem}
{
\begin{remark}\label{rem-integrand-infty}
	Note that the set of integrands $\mathcal{H}^\infty(H,K)$ is large for natural choices of prediction sets. Indeed, if, e.g., $\Xi\subseteq\{\omega\in\Omega : \langle\omega\rangle_T\leq c\}$ for some constant $c>0$, then clearly $\|F\|_{\mathcal{H}^\infty(H,K)}\leq c^{\frac{1}{2}} \sup_{t\in[0,T],\omega\in\Xi} \|F_t(\omega)\|_{L(H,K)}^2$.
	In particular, as every bounded, adapted, and c\`adl\`ag function $F\colon [0,T]\times \Omega\to L(H,K)$ can be approximated uniformly by simple integrands, it then follows that  $F\in \mathcal{H}^\infty(H,K)$.
	
	Moreover, if we 
	require $\Xi\subseteq \Xi_c$, where
	$\Xi_c\subseteq \Omega$  is the prediction set for which there exists a constant $c\geq0$ such that
	\begin{align}
	\label{Xi-c}
	\Xi_c=\bigg\{\omega\in C([0,T],H) : 
	\begin{array}{l}
	\omega \text{ is H\"older continuous and }\langle \omega\rangle \text{ is}\\
	\text{absolutely continuous with }d\langle \omega\rangle/dt \leq c
	\end{array}
	\!\bigg\},	
	\end{align}
	then $\|F\|_{\mathcal{H}^\infty(H,K)}^2\leq  \sup_{\omega\in\Xi_c} (c\int_0^T \|F_t(\omega)\|_{L(H,K)}^2\,dt)^{\frac{1}{2}}$.
	In particular, 
	$\mathcal{H}^\infty(H,K)$ contains for instance all deterministic $L^2(dt)$-Borel functions $F$.
\end{remark}}
\begin{remark}
	If $K$ is a general (not finite dimensional) Hilbert space,
	then the stochastic integral $(F\cdot S)$ exists for every $F\in\mathcal{H}^\infty(H,K)$. However, 
	it remains open whether it has a continuous modification. We refer to Remark~\ref{rem:finite-dim2} for further details.
	\end{remark}


\begin{remark}
	Throughout this paper we work with the real-valued quadratic variation
	$\langle S\rangle$ of the $H$-valued processes $S$.
	However, for $K=\mathbb{R}$ one can instead consider the tensor-valued process 
	$\tlangle S\trangle$ defined by
	\[\tlangle\omega\trangle_t:= \lim_{m\to\infty} \sum_{k=1}^\infty 
	\Big(\omega\big(\sigma^m_{k}(\omega)\wedge t\big) 
	-\omega\big(\sigma^m_{k-1}(\omega)\wedge t\big)\Big)\otimes \Big(\omega\big(\sigma^m_{k}(\omega)\wedge t\big) 
	-\omega\big(\sigma^m_{k-1}(\omega)\wedge t\big)\!\Big)\]
	where $\sigma^m_k(\omega):=\inf\left\{t\geq \sigma^m_{k-1}(\omega) : \|\omega(t)-\omega(\sigma^m_{k-1}(\omega))\|_H \ge 2^{-m}\right\}$ with $\sigma^m_0(\omega)=0$ for all $k,m\in\mathbb{N}$, and where $\otimes$ denotes the tensor product.
	Then the processes $\tlangle S \trangle$ and 
	$\mathbb{S}:=S\otimes S - \tlangle S \trangle$ take values in the 
	tensor space $H\otimes H$.
	In this setting, $\mathcal{E}(\cdot)$ can be defined as before with the difference that 
	the integrands $G^n$ are elements of $\mathcal{H}_s(H\otimes H,\mathbb{R})$.
	In the weak BDG inequality of Theorem \ref{thm:integral}, the term
	$(\|F\|_{L(H,\mathbb{R})}^2\cdot \langle S\rangle)$ has to be
	replaced by ``$(F\otimes F \cdot \tlangle S\trangle)$'',
	see e.g.~\cite[Chapter 20]{metivier1982semimartingales} for more details
	on tensor quadratic variation.
	Note that in case $H=\mathbb{R}^d$ it holds $H\otimes H=\mathbb{R}^{d\times d}$,
	the process 
	$\tlangle S \trangle$ is the symmetric matrix containing the pairwise covariation of all components of $S$, and 
	$(F\otimes F \cdot \tlangle S\trangle)=\sum_{i,j} ( F^iF^j\cdot \langle S^i,S^j\rangle)$.

	Replacing $\langle S\rangle$ by $\tlangle S\trangle$ 
	might be of interest for the following reason:
	The prediction set $\Xi$ may include different predictions for the 
	quadratic variation and covariation of different components of $S$. 
	While this is ignored in
	$\sup_{\omega\in\Xi}(\|F\|_{L(H,K)}^2\cdot \langle S\rangle)_T(\omega)$,
	it is incorporated in 
	$\sup_{\omega\in\Xi}(F\otimes F \cdot \tlangle S \trangle)_T(\omega)$, and the integral $(F\cdot S)$ can potentially be defined for a
	larger space of integrands $F$. 
\end{remark}
{
\begin{remark}\label{rem:rough-path}
	Note that the process $\mathbb{S}$ used in Definition~\ref{def:second-vovk} to define our stochastic integral also appears in rough path theory. 
	Indeed, the iterated integral $\frac{1}{2}(S\otimes S - \tlangle S \trangle)$ is a reduced rough path which allows to define pathwise integrals of gradient 1-forms, see \cite[Section~5]{FrizHairer2014}.
	Nevertheless, the construction of the corresponding integrals are different. The limit procedure with respect to the outer measure allows us to obtain  a larger class of integrands, see Remark~\ref{rem-integrand-infty}.
	However, while our integrals are only defined for typical paths, rough path theory allows to construct integrals pathwise for regular enough integrands.
\end{remark}}
To be able to not only define stochastic integrals for typical paths, but also solve  stochastic differential equations, we need to control the quadratic variation of typical paths.
{To that end, for the rest of this section, we work with the particular prediction set $\Xi_c\subseteq \Omega$
	defined in \eqref{Xi-c}.
}
For any $F\colon [0,T]\times \Omega  \mapsto L(H,K)$, set
\begin{equation*}
\|F\|_{\mathcal{H}^2(H,K)}:=\bigg(\int_0^T\mathcal{E}\big(\|F_t\|_{L(H,K)}^2\big)\,dt\bigg)^{\frac{1}{2}}.
\end{equation*}
For the prediction set $\Xi_c$ 
it turns out that stochastic integrals can be defined 
for integrands lying in the set
\[ \mathcal{H}^2(H,K)
:=\left\{ F\colon\Omega\times[0,T]\to L(H,K) : 
\begin{array}{l}
\|F\|_{\mathcal{H}^2(H,K)}<+\infty, \text{ and there exists a}\\
\text{sequence } (F^n) \text{ in } \mathcal{H}_{s,c}(H,K)\text{ such that}\\
\|F-F^n\|_{\mathcal{H}^2(H,K)} \to 0
\end{array}\!\right\},\]
where $\mathcal{H}_{s,c}(H,K)$ is the set of those $\sum_{n=0}^\infty f_n1_{( \tau_n,\tau_{n+1}]}\in\mathcal{H}_s(H,K)$ such that the stopping times $(\tau_n)$ are deterministic and 
$f_n\colon\Omega\to L(H,K)$ is continuous for each $n$.

Note that $\mathcal{H}^2(H,K)$ is a linear space. More precisely, the following result holds true.

\begin{theorem}
\label{thm:integral.Xi.has.assumptions}
{Let  $\Xi\equiv\Xi_c$ be the prediction set defined in \eqref{Xi-c}}
	and let $F \in \mathcal{H}^2(H,K)$.	Then the stochastic integral 
	\[(F\cdot S)\colon \Omega\to C([0,T],K) \]
	exists and satisfies the following weak BDG-type inequality
	\[\mathcal{E}\Big( \sup_{t\in[0,T]} \|(F\cdot S)_t\|_K^2 \Big)
	\leq 4c \int_0^T \mathcal{E}(\|F_t\|_{L(H,K)}^2)dt
	=4c \|F\|_{\mathcal{H}^2(H,K)}^2.\]
	Moreover, it coincides with the classical stochastic integral under every martingale measure
	$P\in\mathcal{M}(\Xi_c)$.
	In addition, if $f\colon [0,T]\times K\to L(H,K)$ is Lipschitz continuous
  then the map $\Omega \times [0, T ]\ni (\omega,t)\mapsto f (t, (F \cdot S)_t (\omega)) 
  \in L(H, K)$ is an element of $\mathcal{H}^2(H,K)$.
\end{theorem}
{
\begin{remark}\label{rem:integrand-2}
	Already under some mild regularity assumptions on $F\colon [0,T]\times \Omega\to L(H,K)$, we have   $F\in\mathcal{H}^2(H,K)$.
	More precisely, the following holds true.  Let $F\colon [0,T]\times \Omega\to L(H,K)$ be continuous such that $\|F_t-F_s\|_{L(H,K)}\leq \rho(|t-s|) (1+\sup_{r\in[0,T]}\|S_r\|_H^p)$ for some $p\in[1,\infty)$ and some continuous function $\rho$ which satisfies $\rho(0)=0$. 
	Then $F\in\mathcal{H}^2(H,K)$. We provide its proof in Subsection~\ref{subsec:proof-integral-ass}.
\end{remark}
}
%
%
To be able to define a notion of a solution of a stochastic differential equation for typical paths,
let $A\colon [0,T]\times\Omega\to\mathbb{R}$ be a process such that 
$\omega\mapsto A_t(\omega)$ is continuous for all $t$ and 
$t\mapsto |A(\omega)_t|$ is absolutely continuous with
$d |A|(\omega)/dt\le c$ for all $\omega\in\Omega$.
Moreover, let 
$\mu\colon [0,T]\times K\to L(\mathbb{R},K)$ and $\sigma\colon [0,T]\times K\to L(H,K)$
be two functions which satisfy the following.
\begin{assumption}\label{ass:SDE}
There is a constant $L>0$ such that for all $k,k^\prime\in K$ and $t,t^\prime\in[0,T]$ we have that
\begin{align}
\label{eq:Lipschitz-coeff}
\begin{split}
\|\sigma(t,k)-\sigma(t^\prime,k^\prime)\|_{L(H,K)}&\le L(|t-t^\prime|+\|k-k'\|_K),\\
\|\mu(t,k)-\mu(t^\prime,k^\prime)\|_{L(\mathbb{R},K)}&\le L(|t-t^\prime|+\|k-k'\|_K). 
\end{split}
\end{align}
\end{assumption}

Then we can state our third main result stating the existence of solutions of stochastic differential equations for typical paths.

\begin{theorem}
\label{thm:sde}
{Let  $\Xi\equiv\Xi_c$ be the prediction set defined in \eqref{Xi-c}} and  assume that 
 Assumption~\ref{ass:SDE} holds. Moreover, let $x_0 \in K$. Then there exists a unique (up to typical paths) $X\colon\Omega\to C([0,T],K)$ such that 
$X\in\mathcal{H}^2(K,\mathbb{R})$ and $X$ solves the SDE
\[ dX_t = \mu(t,X_t)\,dA_t + \sigma(t,X_t)\,dS_t,\quad X_0=x_0, \]
 i.e.~\[X_t=x_0+(\mu(\cdot,X) \cdot A)_t+ (\sigma(\cdot,X)\cdot S)_t\]
for typical paths.
\end{theorem}

For the precise definition of $(\mu(\cdot,X) \cdot A)$ and $(\sigma(\cdot,X)\cdot S)$ see 
Lemma \ref{lem:intA} \& Remark \ref{rem2:identification} and Lemma~\ref{le:f-Lip-nice} \& Remark~\ref{rem:identification}, respectively.

\begin{remark}\label{rem:explanationPart1}
We point out that 
with our methods we cannot
solve SDEs for typical paths on the space $\mathcal{H}^\infty(H,K)$ instead of $\mathcal{H}^2(H,K)$, {even when $\Xi\equiv\Xi_c$}. 
The reason is that the corresponding norm $\|\cdot\|_{\mathcal{H}^\infty(H,K)}$ is too strong to obtain a (similar) result that if $F \in \mathcal{H}^\infty(H,K)$ and $f\colon [0,T]\times K\to L(H,K)$ is Lipschitz continuous,	
then $f(\cdot,(F\cdot S))\in\mathcal{H}^\infty(H,K)$. But such a relation is the key property necessary to solve SDEs. We refer to Lemma~\ref{le:f-Lip-nice} 
for further details.
\end{remark}
{The main reason why we restrict ourselves in Theorem~\ref{thm:integral.Xi.has.assumptions} and Theorem~\ref{thm:sde} is the following duality result going back  to \cite{bartl2017pathwise}, which is heavily used in the respective proofs.
\begin{theorem}\label{thm:dual-Xi-c}
Let $\Xi\equiv\Xi_c$ be the prediction set defined in \eqref{Xi-c}.
	Then for every $X\colon C([0,T],H)\to[0,+\infty]$ which is the pointwise limit of an increasing sequence $X_n\colon\Omega\to[0,+\infty]$ of upper semicontinuous functions one has
	\begin{equation}
	\label{eq:duality-Xi-c}
	\mathcal{E}(X)=\sup_{P\in\mathcal{M}(\Xi_c)} E_P[X].
	\end{equation}
	In particular, the duality \eqref{eq:duality-Xi-c} holds for every nonnegative upper or lower semicontinuous function $X\colon C([0,T],H)\to[0,+\infty]$. 
\end{theorem}
\begin{remark}\label{rem:quasi-sure}
	From the duality result in Theorem~\ref{thm:dual-Xi-c}, we see that $\mathcal{E}(N)=0$ if and only if $P(N)=0$ for all $P \in  \mathcal{M}(\Xi_c)$.
	Moreover, the duality result shows that the stochastic integral defined in Theorem~\ref{thm:integral.Xi.has.assumptions} and the corresponding solution of the 
	stochastic differential equation in Theorem~\ref{thm:sde} provide a way to solve the aggregation problem appearing in the quasi-sure setting with respect to $\mathcal{M}(\Xi_c)$ $($see, e.g., \cite{Nutz12}$)$, which has particular applications in financial mathematics when model uncertainty occurs. 
\end{remark}
\begin{remark}\label{rem:gen-Xi-c}
	In fact, Theorem~\ref{thm:integral.Xi.has.assumptions} and Theorem~\ref{thm:sde} could have been stated for more general prediction sets $\Xi\subseteq \Xi_c$ satisfying the conditions imposed  in Theorem~\ref{thm:dual} to obtain the duality result analogous to Theorem~\ref{thm:dual-Xi-c}. However, since the conditions are rather technical and our goal is to include as many paths as possible, we decided to state Theorem~\ref{thm:integral.Xi.has.assumptions} and Theorem~\ref{thm:sde} with respect to $\Xi_c$.
\end{remark}
}
\section{Proofs of our main results}
\label{sec:Proof}

\subsection{Properties of $\mathcal{E}(\cdot)$}
In this subsection, we analyze  properties of the outer measure $\mathcal{E}$ which will be crucial to define stochastic integrals and solutions of stochastic differential equation for typical paths.  Throughout this subsection, we work with the conventions 
$0\cdot(+\infty)=(+\infty)\cdot 0= 0$ and $+\infty-\infty=-\infty+\infty=+\infty$.
First, observe that directly from its definition, the outer measure $\mathcal{E}$ is sublinear, positive homogeneous, and satisfies $\mathcal{E}(\lambda)\leq \lambda$ for all $\lambda \in [0,+\infty)$.  In addition, $\mathcal{E}$ satisfies the following properties. 

\begin{proposition}
	\label{prop:expectation.is.sigma.subadd}
	The functional $\mathcal{E}$ is countably subadditive, i.e.
	\[\mathcal{E}\Big({\textstyle\sum\limits_{n=1}^\infty} X_n\Big)
	\leq \sum_{n=1}^\infty \mathcal{E}(X_n)\]
	and satisfies
	\[\mathcal{E}\Big(\big({\textstyle\sum\limits_{n=1}^\infty} X_n \big)^2\Big)^{\frac{1}{2}}
	\leq\sum_{n=1}^\infty \mathcal{E}\big(X_n^2\big)^\frac{1}{2}\]
	for every sequence 
	$X_n\colon\Omega\to[0,+\infty]$, $n \in \mathbb{N}$. Furthermore, it fulfills  the Cauchy--Schwarz inequality
	\[
	\mathcal{E}(|X| |Y|)\leq \mathcal{E}(X^2)^\frac{1}{2} \mathcal{E}(Y^2)^\frac{1}{2} 
	\]
	for all $X,Y\colon\Omega\to[-\infty,+\infty]$.
\end{proposition}
\begin{proof}
The proof of countable subadditivity is the same as in \cite[Lemma 4.1]{Vovk12} and \cite[Lemma 2.3]{PerkowskiPromel16}. However, due to the different setting and in order to be self contained,
we provide a proof. Without loss of generality assume that $\sum_n \mathcal{E}(X_n)<+\infty$.
Fix $\varepsilon>0$, a sequence $(c_n)$ in $(0,+\infty)$ such that $\sum_n c_n=\varepsilon$,
and let $\lambda_n:=\mathcal{E}(X_n)+c_n$, as well as
$\lambda:=\sum_n \lambda_n$.
Then, by definition of $\mathcal{E}(X_n)$, for every $n$ there are two sequences of simple integrands 
$(F^{n,m})_m$ and $(G^{n,m})_m$ such that 
\[
\lambda_n+(F^{n,m}\cdot S)_t+(G^{n,m}\cdot \mathbb{S})_t\geq 0\quad\text{for all $t\in[0,T]$ and $m$}\]
and
\[
\lambda_n+\liminf_m\big( (F^{n,m}\cdot S)_T+(G^{n,m}\cdot \mathbb{S})_T\big) \geq X_n.
\]
Now define the simple integrands $F^m:=\sum_{n\le m} H^{n,m}$ and $G^m:=\sum_{n\le m} G^{n,m}$ for each $m$.
Then $\lambda+(F^m\cdot S)_t+(G^m\cdot \mathbb{S})_t\geq 0$ for all $m$,
and superadditivity of $\liminf$ implies for every $k\in\mathbb{N}$ that
\begin{align*}
&\lambda + \liminf_m \big( (F^m\cdot S)_T+(G^m\cdot \mathbb{S})_T\big)\\
&=\liminf_m\Big( {\textstyle\sum\limits_{n\le m}}\big(\lambda_n +  (F^{n,m} \cdot S)_T+(G^{n,m}\cdot \mathbb{S})_T\big)\Big)\\
&\geq  \sum_{n \le k} \liminf_m\Big(\lambda_n + 
(F^{n,m} \cdot S)_T+(G^{n,m}\cdot \mathbb{S})_T\Big)
\geq \sum_{n\le k} X_n.
\end{align*}
 Passing to the limit in $k$ yields
$\mathcal{E}(\sum_n X_n)\leq \lambda=\sum_n \mathcal{E}(X_n)+\varepsilon$. Since $\varepsilon>0$ was arbitrary, we obtain the first inequality.

As for H\"older's inequality, let $X,Y\colon\Omega\to[-\infty,+\infty]$. First, we assume
that $\mathcal{E}(X^2)<+\infty$ and $\mathcal{E}(Y^2)<+\infty$.
If $\mathcal{E}(X^2)=0$ or $\mathcal{E}(Y^2)=0$, then the pointwise estimate 
$|X| |Y|\leq \frac{\alpha X^2}{2}+\frac{Y^2}{2\alpha}$ for all $\alpha>0$,
together with  sublinearity and monotonicity of $\mathcal{E}$ yields 
\[\mathcal{E}(|X| |Y|)\leq \frac{\alpha}{2}\mathcal{E}(X^2)+\frac{1}{2\alpha}\mathcal{E}(Y^2)\]
so that $\mathcal{E}(|X| |Y|)=0$. If $\mathcal{E}(X^2)>0$ and $\mathcal{E}(Y^2)>0$, then the previous inequality applied to $\tilde X:= X/\mathcal{E}(X^2)^\frac{1}{2}$ and  $\tilde Y:= Y/\mathcal{E}(Y^2)^\frac{1}{2}$ with $\alpha=1$ leads to 
\[\frac{\mathcal{E}(|X| |Y|)}{ \mathcal{E}(X^2)^\frac{1}{2} \mathcal{E}(Y^2)^\frac{1}{2}}=\mathcal{E}(|\tilde X| |\tilde Y|)\leq \frac{\mathcal{E}(\tilde X^2)}{2}+\frac{\mathcal{E}(\tilde Y^2)}{2}= 1.\]
Second, if $\mathcal{E}(X^2)=0$ and $\mathcal{E}(Y^2)=+\infty$ the first part implies
$0\le\mathcal{E}(|X|)\leq \mathcal{E}(X^2)^\frac{1}{2}$, i.e.~$\mathcal{E}(|X|)=0$. Therefore, the pointwise inequality $|X| |Y| \leq \sum_{n=1}^\infty |X|$ together with the countable subadditivity of $\mathcal{E}$ yields that
\[\mathcal{E}\big(|X| |Y|\big)\le \mathcal{E}\Big({\textstyle\sum\limits_{n=1}^\infty} |X|\Big)\le \sum_{n=1}^\infty \mathcal{E}\big( |X|\big)=0.\]

To show the second statement, let $(X_n)$ be a family of functions $X_n\colon\Omega\to[0,+\infty]$ which is at most countable. By the previous steps we have
	\begin{align*}
	\mathcal{E}\Big( \big({\textstyle\sum\limits_{n}} X_n \big)^2 \Big)&= \mathcal{E}\Big( {\textstyle\sum\limits_{n,m}} X_n X_m \Big)
	\leq \sum_{n,m}  \mathcal{E} (  X_n X_m ) \\
	&\leq \sum_{n,m}  \mathcal{E}(X_n^2)^{\frac{1}{2}} \mathcal{E}(X_m^2)^{\frac{1}{2}}
	=\Big(\sum_n  \mathcal{E}(X_n^2)^{\frac{1}{2}}\Big)^2.
	\end{align*}
	It remains to take the root.
\end{proof}

\begin{proposition}
\label{prop:L2-Banach-Banach}
	Given an arbitrary Banach space $(B,\|\cdot\|_B)$, we define
	$\|X\|:=\mathcal{E}(\|X\|_B^2)^{\frac{1}{2}}$ for $X\colon\Omega\to B$.
	Then the following hold:
	\begin{enumerate}[(i)]
		\item\label{seminorm} The functional $\|\cdot\|$ is a semi-norm, i.e.~it only takes non-negative values, is absolutely homogeneous,
		and satisfies the triangle inequality. 
		\item \label{lem:complete}
		Every Cauchy sequence $X_n\colon\Omega\to B$, $n\in \mathbb{N}$, w.r.t.~$\|\cdot\|$ has a limit $X\colon \Omega\to B$, i.e.~$\|X_n - X\|\to 0$, and there is a subsequence $(n_k)$ such that $X_{n_k}(\omega)\to X(\omega)$ for typical paths.
	\end{enumerate}
\end{proposition}
\begin{proof}
	It is clear that $\|\cdot\|$ only takes values in $[0,+\infty]$ and is absolutely homogeneous.
	To show the triangle inequality, let $X,Y\colon\Omega\to B$. 
	It follows from Proposition \ref{prop:expectation.is.sigma.subadd} that
	\begin{align*}
	\|X+Y\|
	&\le \mathcal{E}\big((\|X\|_B+\|Y\|_B)^2\big)^{\frac{1}{2}}  \\
	&\le \mathcal{E}\big(\|X\|_B^2\big)^{\frac{1}{2}} + \mathcal{E}\big(\|Y\|_B^2\big)^{\frac{1}{2}} 
	= \|X\| + \|Y\|.
	\end{align*}

	To see that \eqref{lem:complete} holds true, let $(X_n)$ be a Cauchy sequence and  
	choose a subsequence $(X_{n_k})$ such that $\|X_{n_{k+1}}-X_{n_{k}}\|\leq 2^{-k}$. 
	By Proposition \ref{prop:expectation.is.sigma.subadd} it holds
	\[ \mathcal{E}\Big(\big({\textstyle\sum\limits_{k}} \|X_{n_{k+1}}-X_{n_k}\|_B\big)^2\Big)^{\frac{1}{2}}
	\leq  \sum_{k} \|X_{n_{k+1}}-X_{n_k}\|
	<+\infty. \]
	This implies that the set $N:=\{\sum_{k} \|X_{n_{k+1}}-X_{n_k}\|_B=+\infty\}$
	satisfies $\mathcal{E}(1_N)=0$. 
	As $B$ is complete, for every $\omega\in N^c$, the sequence $(X_{n_{k}}(\omega))_k$ has a limit.
	Therefore, $X:=\lim_k X_{n_k}1_{N^c}$ is a mapping from $\Omega$ to $B$ and 
	Proposition \ref{prop:expectation.is.sigma.subadd} yields  
	\[ \|X-X_{n_k}\|
	\leq \mathcal{E}\Big(\big( {\textstyle\sum\limits_{l\geq k}} \|X_{n_{l+1}}-X_{n_l}\|_B\big)^2 \Big)^{\frac{1}{2}}
	\leq \sum_{l\geq k} \| X_{n_{l+1}}-X_{n_l} \|
	\leq 2^{-k+1}
	\to 0\]
	as $k$ tends to infinity. 
	Since $(X_n)$ is a Cauchy sequence, the triangle inequality shows that 
	$\|X-X_n\|\leq \|X-X_{n_k}\|+\|X_{n_k}-X_n\|\to 0$ as $k,n\to\infty$.
\end{proof}

\begin{lemma}
\label{lem:weak.duality}
	For every $P\in\mathcal{M}(\Xi)$ and every measurable $X\colon \Omega\to[0,+\infty]$
	one has $E_P[X]\leq\mathcal{E}(X)$.
\end{lemma}
\begin{proof}
Fix $P\in\mathcal{M}(\Xi)$ and $X\colon \Omega\to[0,+\infty]$ measurable.

First, if $F\in\mathcal{H}_s(H,\mathbb{R})$ and 
	$G\in\mathcal{H}_s(\mathbb{R},\mathbb{R})$ are of the form
	\[F=\sum_{n\leq N} f_n1_{(\tau_n,\tau_{n+1}]}\quad\mbox{and}\quad
	G=\sum_{n\leq N} g_n1_{(\tau_n,\tau_{n+1}]}\] such that $\sup_{n\le N} \|f_n\|_H<\infty$ and $\sup_{n \le N} |g_n|<\infty$ for some $N\in\mathbb{N}$, then the process $(F\cdot S)+(G\cdot \mathbb{S})$ is a continuous martingale (the martingale property follows e.g.~by approximating $f_n$ $P$-a.s.~by 
	functions with finite range and dominated convergence).

    Second, let $\lambda\geq 0$, $F=\sum_{n} f_n1_{(\tau_n,\tau_{n+1}]}$ in $\mathcal{H}_s(H,\mathbb{R})$, and $G=\sum_{n} g_n1_{(\tau_n,\tau_{n+1}]}$ in $\mathcal{H}_s(\mathbb{R},\mathbb{R})$ be 
	such that $\lambda + (F\cdot S)_t + (G\cdot \mathbb{S})_t\geq 0$ on $\Xi$ for all $t$.
	Define the stopping times
	\[\sigma_m:=\inf\{t\geq 0 : \|S_t\|_H\geq m \text{ or } \|F_t\|_{L(H,\mathbb{R})}\geq m\}\wedge \tau_m\] and similarly, $\tilde{\sigma}_m:=\inf\{t\geq 0 : |\mathbb{S}_t|\geq m \text{ or } \|G_t\|_{L(\mathbb{R},\mathbb{R})}\geq m\}\wedge \tau_m$. Note that for every 
	$\omega$ there is $m=m(\omega)$ such that $\sigma_m(\omega)=\tilde{\sigma}_m(\omega)=T$.
	Defining $F^m:= F1_{(0,\sigma_m]}$ in $\mathcal{H}_s(H,\mathbb{R})$
	and $G^m:= G1_{(0,\tilde{\sigma}_m]}$ in $\mathcal{H}_s(\mathbb{R},\mathbb{R})$ for every $m$
	one observes that $(F^m\cdot S)_t= (F\cdot S)_{t\wedge \sigma_m}$
	and $(G^m\cdot\mathbb{S})_t= (G\cdot \mathbb{S})_{t\wedge \sigma_m}$ 
	converge pointwise to $(F\cdot S)_t$ and $(G\cdot \mathbb{S})_t$, respectively.
	However, since $(F^m\cdot S)$ and $(G^m\cdot\mathbb{S})$ are martingales by the first step and 
	$\lambda +(F^m\cdot S)_t+(G^m\cdot\mathbb{S})_t\geq 0$ on $\Xi$ for all $t$, 
	it follows from Fatou's lemma that $(F\cdot S)+(G\cdot\mathbb{S})$ is a supermartingale.

	Finally, let $\lambda\geq 0$, $(F^n)$ a sequence in $\mathcal{H}_s(H,\mathbb{R})$, and 
	$(G^n)$ a sequence in $\mathcal{H}_s(\mathbb{R},\mathbb{R})$ such that 
        $\lambda+(F^n\cdot S)_t+(G^n\cdot \mathbb{S})_t\geq 0$ on $\Xi$ for all $t$ and
	$\lambda+ \liminf_n\big( (F^n\cdot S)_T +(G^n\cdot \mathbb{S})_T \big)\geq X$ on $\Xi$.
	Since $(F^n\cdot S)+(G^n\cdot\mathbb{S})$ is a supermartingale by the previous arguments, it follows from 
	Fatou's lemma that
	\begin{align*}
	E_P[X]
	&\leq E_P\Big[\lambda+\liminf_n ( (F^n\cdot S)_T+(G^n\cdot\mathbb{S})_T ) \Big]\\
	&\leq  \liminf_n E_P[ \lambda +   (F^n\cdot S)_T+(G^n\cdot\mathbb{S})_T ] 
	\leq \lambda.
	\end{align*}
	As $\lambda$ was arbitrary, this shows $E_P[X]\leq\mathcal{E}(X)$.
\end{proof}

\subsection{Proof of Theorem~\ref{thm:integral}}
The goal of this subsection is to prove Theorem~\ref{thm:integral}. Lemma~\ref{lem:integrals.squared}, though of elementary nature, is the key observation in what follows. More precisely, it is exactly in this Lemma that we see that adding the second order, i.e. integrals with respect to $\mathbb{S}$ in the definition of $\mathcal{E}$,  leads to a simple It\^o isometry and, with the help of a pathwise inequality, to a BDG-inequality. This, in turn, allows us to directly define stochastic integrals for typical paths which are continuous.
\begin{lemma}
	\label{lem:integrals.squared}
	For every simple integrand $F\in\mathcal{H}_s(H,K)$ there exists $\tilde{F} \in \mathcal{H}_s(H,\mathbb{R})$ such that
	for all $t\in[0,T]$ we have
	\[\|(F\cdot S)_t\|^2_K \leq (\tilde{F}\cdot S)_t+ (\|F\|^2_{L(H,K)}\cdot \mathbb{S})_t 
		+ (\|F\|^2_{L(H,K)}\cdot \langle S\rangle)_t.\]
	In particular, the following weak It\^o isometry holds true. For every $F\in\mathcal{H}_s(H,K)$, it holds that
	\[\mathcal{E}(\|(F\cdot S)_T\|^2_K)
	\leq \sup_{\omega\in \Xi}(\|F\|^2_{L(H,K)}\cdot \langle S\rangle)_T(\omega)
	=	\|F\|_{\mathcal{H}^\infty(H,K)}^2.\]
\end{lemma}
\begin{proof}
	Fix a simple integrand $F=\sum_nf_n1_{(\tau_n,\tau_{n+1}]}\in\mathcal{H}_s(H,K)$. By definition, when denoting $f^*_n$ the adjoint operator of $f_n$, we see that
	\begin{align*}
	\|(F\cdot S)_t\|^2_K
	&=\sum_{n,m}\Big\langle f_m (S^t_{\tau_{m+1}}-S^t_{\tau_m}), f_n (S^t_{\tau_{n+1}}-S^t_{\tau_n})\Big\rangle_K\\
	& =2\sum_{m<n}\Big\langle f_n^\ast f_m (S^t_{\tau_{m+1}}-S^t_{\tau_m}),  S^t_{\tau_{n+1}}-S^t_{\tau_n} \Big\rangle_H 
	 +\sum_n \| f_n (S^t_{\tau_{n+1}}-S^t_{\tau_n}) \|_K^2,
	\end{align*}
	where all sums are finite by definition of simple integrands.
	Using the inequality 
	$\| f_n (S^t_{\tau_{n+1}}-S^t_{\tau_n}) \|_K^2 
	\le \|f_n\|^2_{L(H,K)} \|S^t_{\tau_{n+1}}-S^t_{\tau_n}\|^2_H$
	and since 
   \begin{align*}
	\|S^t_{\tau_{n+1}}-S^t_{\tau_n}\|^2_H&= \|S^t_{\tau_{n+1}}\|^2_H- \|S^t_{\tau_{n}}\|^2_H -2 \langle S^t_{\tau_n}, S^t_{\tau_{n+1}}-S^t_{\tau_n}\rangle_H \\
	&= \mathbb{S}^t_{\tau_{n+1}} - \mathbb{S}^t_{\tau_{n}} + \langle S\rangle^t_{\tau_{n+1}}- \langle S\rangle^t_{\tau_{n}}  -2 \langle S^t_{\tau_n}, S^t_{\tau_{n+1}}-S^t_{\tau_n}\rangle_H 
	\end{align*}
	it follows that
	\begin{align*}
	\|(F\cdot S)_t\|^2_K
	&\le 2\Big\langle \sum_{m<n} f_n^\ast f_m (S^t_{\tau_{m+1}}-S^t_{\tau_m}) - \|f_n\|^2_{L(H,K)} S^t_{\tau_n},  S^t_{\tau_{n+1}}-S^t_{\tau_n} \Big\rangle_H \\
	&\qquad +\sum_n  \|f_n\|^2_{L(H,K)}(\mathbb{S}^t_{\tau_{n+1}}-\mathbb{S}^t_{\tau_{n}})
	+\sum_n \|f_n\|^2_{L(H,K)}(\langle S\rangle^t_{\tau_{n+1}}-\langle S\rangle^t_{\tau_{n}})
	\\&=(\tilde{F}\cdot S)_t+ (\|F\|^2_{L(H,K)}\cdot \mathbb{S})_t + (\|F\|^2_{L(H,K)}\cdot \langle S\rangle)_t,
	\end{align*}
	where 
	\[\tilde{F}:=2\sum_n \Big(\sum_{m<n} f_n^\ast f_m (S_{\tau_{m+1}}-S_{\tau_m}) - \|f_n\|^2_{L(H,K)} S_{\tau_n}\Big)1_{(\tau_n,\tau_{n+1}]}
	\in\mathcal{H}_s(H,\mathbb{R}).\]

	For the second claim let 
	$\lambda:=\sup_{\omega\in \Xi}(\|F\|^2_{L(H,K)}\cdot \langle S\rangle)_T(\omega)$.
	Since $\tilde{F}\in\mathcal{H}_s(H,\mathbb{R})$, $\|F\|^2_{L(H,K)}\in\mathcal{H}_s(\mathbb{R},\mathbb{R})$,
	and 
	\begin{align*}
	&\lambda + (\tilde{F}\cdot S)_t+ (\|F\|^2_{L(H,K)}\cdot \mathbb{S})_t \geq \|(F\cdot S)_t\|_K^2
	\geq 0 \quad\text{for all }t\\
	&\lambda + (\tilde{F}\cdot S)_T+ (\|F\|^2_{L(H,K)}\cdot \mathbb{S})_T \geq \|(F\cdot S)_T\|_K^2,
	\end{align*}
	it follows that 
	$\mathcal{E}(\|(F\cdot S)_T\|_K^2)\leq \lambda=\|F\|_{\mathcal{H}^\infty(H,K)}^2$.
\end{proof}

We first recall an inequality from \cite{acciaio2013trajectorial} which turns out to be crucial for the proof of the weak BDG inequality. 
The connection between pathwise inequalities as in \eqref{inequality:martingale} and
martingale inequalities are studied in 
\cite{acciaio2013trajectorial,beiglbock2014martingale,beiglbock2015pathwise}.

\begin{lemma}
	\label{lem:bgd.path}
	Let $N\in\mathbb{N}$. Then for every $(x_0,\dots,x_N)\in\mathbb{R}^{N+1}$ it holds that
	\begin{equation}\label{inequality:martingale}
	\max_{0\leq n\leq N} x_n^2\leq 4 x_N^2-4\sum_{n=0}^{N-1}\big(\max_{0\leq i\leq n} x_i\big)(x_{n+1}-x_n).
	\end{equation}
\end{lemma}
\begin{proof}
	This is \cite[Proposition 2.1]{acciaio2013trajectorial} and the remark afterwards.
\end{proof}
\begin{proposition}[Weak BDG inequality for simple integrands]
\label{prop:bdg.calE}
	Assume that $K$ is finite dimensional. Then for every $F\in\mathcal{H}_s(H,K)$ one has that
	\[ \mathcal{E}\Big( \sup_{t\in[0,T]} \|(F\cdot S)_t\|_K^2 \Big)
	\leq 4 \sup_{\omega\in\Xi}\big( \|F\|_{L(H,K)}^2\cdot \langle S\rangle\big)_T(\omega)
	=4 \|F\|_{\mathcal{H}^\infty(H,K)}^2.\]
\end{proposition}
\begin{proof}
	First assume that $K=\mathbb{R}$ and fix a simple integrand
	$F=\sum_{n} f_n1_{(\tau_n,\tau_{n+1}]}
	=\sum_n F_{\tau_{n+1}}1_{(\tau_n,\tau_{n+1}]}\in\mathcal{H}_s(H,K)$. 
	Let $\sigma_n^m$, $n,m \in \mathbb{N}$, be stopping times such that
	\[\{\tau_n(\omega):n\in\mathbb{N}\}\subset \{ \sigma^m_n(\omega) : n\in\mathbb{N}\}=: D^m(\omega),\]
	for every $m$ the set  $D^m(\omega)$ is finite and $D^m(\omega)\subset D^{m+1}(\omega)$,
	and the closure of $\cup_m D^m(\omega)$ equals $[0,T]$ for every $\omega\in\Omega$.
	Now fix some $m$ and note that 
	\[F=\sum_n F_{\sigma^m_{n+1}}1_{(\sigma^m_n,\sigma^m_{n+1}]}
	\quad\text{since } (\sigma^m_n) \text{ is finer than }(\tau_n).\]
	Now, for every $m$ define
	\[ \tilde{H}^m:=-4\sum_n \Big( \big(\max_{0\leq i\leq n} (F\cdot S)_{\sigma_i^m}\big)
	F_{\sigma^m_n}	\Big) 1_{(\sigma_n^m,\sigma_{n+1}^m]}  \in \mathcal{H}_s(H,\mathbb{R}).\]
	Then, by Lemma \ref{lem:bgd.path} (applied to ``$x_n=(F\cdot S^t)_{\sigma_n^m}$'') it holds for all $t\in[0,T]$ that
	\begin{align*}
	\max_{n\in\mathbb{N}} (F\cdot S^t)_{\sigma_n^m}^2
  	&\leq 4 (F\cdot S^t)_T^2 - 4\sum_{n\in\mathbb{N}}  \big(\max_{0\leq i\leq n} (F\cdot S^t)_{\sigma_i^m}\big)
	F_{\sigma^m_{n+1}}(S^t_{\sigma^m_{n+1}}-S^t_{\sigma^m_n})\\
	&= 4(F\cdot S)_t^2+ (\tilde{H}^m\cdot S)_t.	
	\end{align*}
	Furthermore, by Lemma \ref{lem:integrals.squared}, there exists $\tilde{F}\in\mathcal{H}_s(H,\mathbb{R})$ such that for every $t\in[0,T]$ we have 
	\[(F\cdot S)_t^2 \leq (\tilde{F}\cdot S)_t+ (\|F\|^2_{L(H,K)}\cdot \mathbb{S})_t 
		+ (\|F\|^2_{L(H,K)}\cdot \langle S\rangle)_t.\]	
Therefore, setting
	$H^m:=4\tilde{F}+\tilde{H}^m\in\mathcal{H}_s(H,\mathbb{R})$,\, %
	$G:=4\|F\|^2_{L(H,K)}\in\mathcal{H}_s(\mathbb{R},\mathbb{R})$, and 
$\lambda:=4\sup_{\omega\in\Xi} (\|F\|^2_{L(H,K)}\cdot \langle S\rangle)_T(\omega)$, we obtain for every $t\in[0,T]$ that
	\[0\leq \max_{n\in\mathbb{N}} (F\cdot S^t)_{\sigma_n^m}^2
	\leq \lambda + (H^m\cdot S)_t +(G\cdot \mathbb{S})_t. \]
	As $\cup_m\{\sigma^n_m:n\}$ is dense in $[0,T]$, we conclude that 
	\[ \max_{t\in[0,T]} (F\cdot S)_t^2
	=\lim_m \max_{n\in\mathbb{N}} (F\cdot S^T)_{\sigma_n^m}^2
	\leq  \lambda + \liminf_m\big( (H^m\cdot S)_T +(G\cdot \mathbb{S})_T \big). \]
	This shows that $\mathcal{E}(\sup_{t\in[0,T]} (F\cdot S)_t^2)\leq \lambda=4\|F\|_{\mathcal{H}^\infty(H,K)}^2$,
	and concludes the proof for $K=\mathbb{R}$.

	If $K$ is finite dimensional (say with orthonormal basis $\{k_1,\dots,k_d\}$), one has 
	$\|(F\cdot S)_t\|_K^2=\sum_{i=1}^d (F^i\cdot S)_t^2$ with $F^i=\langle F,k_i\rangle_K$.
	Therefore, one can apply the previous step to every $F^i$ to obtain the desired result also when $K$ is finite dimensional.	
\end{proof}

\begin{proof}[\textbf{Proof of Theorem \ref{thm:integral}}]
	Fix $F\in\mathcal{H}^\infty(H,K)$ and a sequence $(F^n)$ in $\mathcal{H}_s(H,K)$ such that
	$\|F-F^n\|_{\mathcal{H}^\infty(H,K)}\to 0$. 
	Now Proposition \ref{prop:bdg.calE} implies that
	\[\mathcal{E}\Big(\sup_{t\in[0,T]}\|(F^n\cdot S)_t-(F^m\cdot S)_t\|_K^2\Big)
	\leq 4\|F^n-F^m\|_{\mathcal{H}^\infty(H,K)}^2.\]
	Therefore Proposition \ref{prop:L2-Banach-Banach} (applied to the Banach space $B:=C([0,T],K)$)
	implies the existence of a limit $(F\cdot S)\colon \Omega\to B$ such that
	\[\mathcal{E}\Big(\sup_{t\in[0,T]}\|(F\cdot S)_t-(F^n\cdot S)_t\|_K^2\Big)
	\to 0. \]

	The proof that (for typical paths) $(F\cdot S)$ does not depend on the choice of the sequence $(F^n)$
	which converges to $F$ and that the BDG inequality extends to $F\in\mathcal{H}^\infty(H,K)$ follows
	from the triangle inequality by standard arguments. 
	Moreover, we derive from the well-known $L^2(P)$-limit procedure for the construction of the classical stochastic integral (see, e.g., \cite{philip2004stochastic}) and Lemma~\ref{lem:weak.duality} that indeed, the constructed  stochastic integral for typical paths coincides with the classical stochastic integral under every martingale measure
	$P\in\mathcal{M}(\Xi)$.
\end{proof}

\begin{remark}
\label{rem:finite-dim2}
By arguing like in the proof of Theorem~\ref{thm:integral}, but using the weak It\^o isometry introduced in Lemma~\ref{lem:integrals.squared} instead of the weak BDG-inequality defined in Proposition~\ref{prop:bdg.calE}, we can define stochastic integrals with respect to integrands $F \in \mathcal{H}^\infty(H,K)$ without imposing that $K$ is finite dimensional. Moreover, using standard arguments involving the triangle inequality, we see that the weak It\^o isometry introduced in Lemma~\ref{lem:integrals.squared} for simple integrands in $\mathcal{H}_s(H,K)$ also holds true for integrands in  $\mathcal{H}^\infty(H,K)$. However, it remains open if the stochastic integral with respect to integrands in $\mathcal{H}^\infty(H,K)$ possess a continuous modification, since the weak It\^o isometry, compared to the weak BDG inequality whose proof depends on the fact that $K$ is finite dimensional, is too weak to guarantee that the sequence of simple integrals converges uniformly.
\end{remark}

\subsection{Duality Result for Second-order Vovk's outer measure}
\label{subsec:duality}
The goal of this subsection is to provide a duality result for the outer measure $\mathcal{E}$.
\begin{theorem}\label{thm:dual}
	Assume that $\Xi=\{\omega\in \Omega : (\omega,\langle\omega\rangle)\in\bar{\Xi} \}$
	for a subset $\bar{\Xi}\subset C([0,T],H)\times C([0,T],\mathbb{R})$ which satisfies
	\begin{itemize}
		\item[(i)]
		$\bar{\Xi}$ is the countable union of compact sets (w.r.t.~uniform convergence),
		\item[(ii)] $\bar{\omega}(\cdot\wedge t)\in\bar{\Xi}$ for all $\bar{\omega}\in\bar{\Xi}$ and $t\in[0,T]$,
		\item[(iii)] $\nu(0)=0$ and $\nu$ is nondecreasing for all $\bar{\omega}=(\omega,\nu)\in\bar{\Xi}$.
	\end{itemize}
	Then for every $X\colon C([0,T],H)\to[0,+\infty]$ which is the pointwise limit of an increasing sequence $X_n\colon\Omega\to[0,+\infty]$ of upper semicontinuous functions one has
	\begin{equation}
	\label{eq:duality}
	\mathcal{E}(X)=\sup_{P\in\mathcal{M}(\Xi)} E_P[X].
	\end{equation}
	In particular, the duality \eqref{eq:duality} holds for every nonnegative upper or lower semicontinuous function $X\colon C([0,T],H)\to[0,+\infty]$. 
\end{theorem}

{Note that the prediction set $\Xi_c$ defined in \eqref{Xi-c} satisfies the conditions (i)--(iii).} Moreover, up to a different admissibility condition in the definition of $\mathcal{E}$, the
statement of Theorem \ref{thm:dual} is similar to \cite[Theorem 2.2]{bartl2017pathwise}. 
However, the additional assumption that $\bar{\Xi}$ contains all stopped paths, can be used as
in the proof of \cite[Theorem~2.1]{bartl2017duality} to deduce Theorem \ref{thm:dual}
in the present setting. Therefore, we only provide a sketch of the proof
with main focus on the admissibility condition.
Let us recall the setting of \cite{bartl2017pathwise}. On 
\[\bar{\Omega}:=C([0,T],H)\times C([0,T],\mathbb{R})\] 
we consider the processes $\bar{S}_t(\bar{\omega}):=\omega(t)$ and 
$\bar{\mathbb{S}}_t(\bar{\omega}):=\|\omega(t)\|_H^2-\nu(t)$ for 
$\bar{\omega}=(\omega,\nu)\in\bar{\Omega}$.
For
$\bar{\Delta}:=\{\bar{\omega}\in\bar{\Omega} : \omega\in \Omega\text{ and }\langle\omega\rangle=\nu\}$, we consider the filtration
$\bar{\mathbb{F}}_+^{\bar{\Delta}}$ defined as the right-continuous version of
$\bar{\mathcal{F}}_t^{\bar{\Delta}}=\sigma(\bar{S}_s,\bar{\mathbb{S}}_s:s\leq t)\vee \sigma( \bar{N}: \bar{N}\subset \bar{\Omega}\setminus\bar{\Delta})$.
A function $\bar{F}$ on $[0,T]\times\bar{\Omega}$ with values in $L(H,\mathbb{R})$ or $L(\mathbb{R},\mathbb{R})$ is called simple if it is of the form
\[\bar{F}_t(\bar{\omega})=\sum_{n\in\mathbb{N}} \bar{f}_n(\bar{\omega})
1_{( \bar{\tau}_n(\bar{\omega}),\bar{\tau}_{n+1}(\bar{\omega})]}(t)\]
for $(t,\bar{\omega})\in [0,T]\times\bar{\Omega}$, where $0=\bar{\tau}_0\leq\dots\leq \bar{\tau}_n\leq\bar{\tau}_{n+1}\leq \dots\leq T$ are $\bar{\mathbb{F}}^{\bar{\Delta}}_+$-stopping times such that 
for each $\bar{\omega}$ only finitely many stopping times are strictly smaller than $T$, and
$\bar{f}_n$ are $\bar{\mathcal{F}}^{\bar{\Delta}}_{\bar{\tau}_n+}$-measurable functions on $\bar\Omega$ with values in $L(H,\mathbb{R})$ and $L(\mathbb{R},\mathbb{R})$, respectively. The function
$\bar{F}$ is called finite simple, if $\bar{\tau}_n=T$  for all $n\geq N$ for some $N\in\mathbb{N}$. Consider the functional
	\[  \bar{\mathcal{E}}(\bar{X}):=
	\inf\left\{ \lambda\geq 0 \,:\, 
	\begin{array}{l}
	\text{there are two simple sequences $(\bar{H}^n)$, $(\bar{G}^n)$ such that}\\ 
	\text{$\lambda+(\bar{H}^n\cdot \bar{S})_t+(\bar{G}^n\cdot \bar{\mathbb{S}})_t\geq 0$ 
	on $\bar{\Delta}\cap\bar{\Xi}$ for all $t$ and $n$, and} \\
	\text{$\lambda+ \liminf_n\big( (\bar{G}^n\cdot\bar{S})_T +(\bar{G}^n\cdot \bar{\mathbb{S}})_T \big)\geq \bar{X}$ on $\bar{\Delta}\cap\bar{\Xi}$}
	\end{array}\!\right\}\]
for $\bar{X}\colon\bar{\Omega}\to[-\infty,+\infty]$.
Further, for a measurable set $\bar{A}\subset\bar{\Omega}$, we denote by
$\bar{\mathcal{M}}(\bar{A})$ the set of all Borel probabilities $\bar{P}$ on $\bar{\Omega}$ such that
$\bar{P}(\bar{A})=1$ and both $\bar{S}$ and $\bar{\mathbb{S}}$ are $\bar{P}$-martingales w.r.t.~the filtration
$\bar{\mathbb{F}}_+^{\bar{\Delta}}$.

The reason to consider the enlarged space $\bar{\Omega}$ is that the duality argument in the 
proof of Theorem \ref{thm:dual} builds on topological arguments and the set $\Xi$ in contrast to $\bar{\Xi}$ is not regular enough. The following transfer principle is the reason why duality on $\Omega$ can be recovered from
duality on the enlarged space $\bar{\Omega}$.

\begin{lemma}
	\label{lem:transfer}
	For every measurable function  $X\colon C([0,T],H)\to[0,+\infty]$ one has
	\[ \mathcal{E}(X)=\bar{\mathcal{E}}(X\circ\bar{S})
	\quad\text{and}\quad
	\sup_{P\in\mathcal{M}(\Xi)} E_P[X]
	=\sup_{\bar{P}\in\bar{\mathcal{M}}(\bar{\Xi})} E_{\bar{P}}[X\circ \bar{S}]. \]
\end{lemma}
\begin{proof}
	The proof of the first equality follows directly from \cite[Lemma~4.5]{bartl2017pathwise}, whereas the second equality is  a direct consequence of \cite[Lemma 4.6]{bartl2017pathwise}.
\end{proof}

\begin{lemma}
	\label{lem:choose.compacts}
	There is an increasing sequence of nonempty compact sets $\bar{\Xi}_n\subset\bar{\Omega}$ such that 
	$\bar{\Xi}=\bigcup_n \bar{\Xi}_n$ and $\bar{\omega}(\cdot\wedge t)\in \bar{\Xi}_n$ for every 
	$\bar{\omega}\in \bar{\Xi}_n$ and $t\in[0,T]$. 
\end{lemma}
\begin{proof}
	The proof is similar to \cite[Lemma 4.5]{bartl2017duality}.
\end{proof}

\begin{lemma}
	\label{lem:integral.over.time.geq0}
	Assume that $\lambda + (\bar{G}\cdot \bar{S})_T+(\bar{H}\cdot\bar{\mathbb{S}})_T\geq 0$ on $\bar{\Xi}$, where  $\bar{G}$ and $\bar{H}$ are
	finite simple integrands.
	Then $\lambda + (\bar{G}\cdot \bar{S})_t+(\bar{G}\cdot\bar{\mathbb{S}})_t\geq 0$ on $\bar{\Xi}$
	for every $t\in[0,T]$.
\end{lemma}
\begin{proof}
	The proof is similar to \cite[Lemma 4.6]{bartl2017duality}.
\end{proof}

We are now ready for the proof of the duality theorem. 

\begin{proof}[\text{Proof of Theorem \ref{thm:dual}}] 
	We first approximate the functional $\bar{\mathcal{E}}$.
	To that end, let $\bar{\Xi}_n$ be the sets of Lemma \ref{lem:choose.compacts} and for $\bar{X}\colon\bar{\Omega}\to [-\infty,+\infty]$ define 
	\[\bar{\mathcal{E}}_n(\bar{X}):=
	\inf\left\{ \lambda\in\mathbb{R}:
	\begin{array}{l}
	\text{there are finite simple integrands $\bar{H}$, $\bar{G}$ and $c\geq 0$ such that}\\ 
	\text{$\lambda+(\bar{H}\cdot \bar{S})_t+(\bar{G}\cdot \bar{\mathbb{S}})_t\geq -c$ on
		$\bar{\Delta}\cap\bar{\Xi}_n$ for all $t,n$, and} \\
	\text{$\lambda+ (\bar{G}\cdot\bar{S})_T +(\bar{G}\cdot \bar{\mathbb{S}})_T \geq \bar{X}$ 
		on $\bar{\Delta}\cap\bar{\Xi}_n$}.
	\end{array}\!\right\}. \] 
	As $\bar{\Xi}_n$ is compact, one can verify that $\bar{\mathcal{E}}_n$ is sufficiently regular
	so that for every upper semicontinuous and bounded function $X\colon \bar{\Omega}\to\mathbb{R}$ we have
	\begin{equation}\label{rep:n}
	\bar{\mathcal{E}}_n(\bar{X})=\sup_{\bar{P}\in\bar{\mathcal{M}}(\bar{\Xi}_n)} E_{\bar{P}}[\bar{X}].
	\end{equation}
	The precise argumentation is given in the steps (a)-(c) of \cite[Theorem 2.2]{bartl2017pathwise}.
	
	Next, let $\bar{X}_n\colon \bar{\Omega}\to[0,+\infty)$, $n \in \mathbb{N}$, be bounded upper semicontinuous functions 
	and $\bar{X}:=\sup_n \bar{X}_n$.  
	Then it holds 
	$\bar{\mathcal{E}}(\bar{X})=\sup_{\bar{P}\in\bar{\mathcal{M}}(\bar{\Xi})} E_{\bar{P}}[\bar{X}]$.
	Indeed, by the weak duality in Lemma \ref{lem:weak.duality},
		using that $\bar{\mathcal{M}}(\bar{\Xi})\supset \bar{\mathcal{M}}(\bar{\Xi}_n)$,
	$\bar{X}\geq \bar{X}_n$ for all $n$, and the representation \eqref{rep:n}
	 one has 
	\begin{align}
	\label{eq:calE.geq.supP.geq.supcalEn}
	\bar{\mathcal{E}}(\bar{X})
	\geq \sup_{\bar{P}\in\bar{\mathcal{M}}(\bar{\Xi})} E_{\bar{P}}[\bar{X}]
	\geq \sup_n \sup_{\bar{P}\in\bar{\mathcal{M}}(\bar{\Xi}_n)} E_{\bar{P}}[\bar{X}_n]
	=\sup_n\bar{\mathcal{E}}_n(\bar{X}_n) .
	\end{align}
	This shows $\bar{\mathcal{E}}(\bar{X})\geq \sup_n\bar{\mathcal{E}}_n(\bar{X}_n)$.
	To prove the reverse inequality, 
	which then implies that all inequalities in \eqref{eq:calE.geq.supP.geq.supcalEn} are actually equalities, one may assume without loss of generality that 
	$m:=\sup_n\bar{\mathcal{E}}_n(\bar{X}_n)<+\infty$.
	Given some fixed $\varepsilon>0$, for every $n$ 
	there exist by definition finite simple integrands $\bar{G}^n$ and $\bar{H}^n$ such that
	\begin{align}
	\label{eq:strategy.for.mathcalEn} 
	m+\varepsilon/2 + (\bar{G}^n\cdot\bar{S})_T+ (\bar{H}^n\cdot\bar{\mathbb{S}})_T\geq \bar{X}_n 
	\quad\text{on } \bar{\Delta}\cap\bar{\Xi}_n.
	\end{align}
	Now define the stopping times
	\[ \bar{\sigma}_n:=\inf\{t\in[0,T] :  m+\varepsilon/2 + (\bar{G}^n\cdot \bar{S})_t+ (\bar{H}^n\cdot\bar{\mathbb{S}})_t=-\varepsilon/2  \}\wedge T. \]  
	Then $(\bar{G}^n\cdot \bar{S})_{\bar{\sigma}_n}=(\tilde{G}^n\cdot \bar{S})_T$
	for the simple integrand $\tilde{G}^n:=\bar{G}^n1_{[0,\bar{\sigma}_n]}$
	and similarly $(\bar{H}^n\cdot \bar{\mathbb{S}})_{\bar{\sigma}_n}=(\tilde{H}^n\cdot \bar{\mathbb{S}})_T$ 
	for $\tilde{H}^n:=\bar{H}^n 1_{[0,\bar{\sigma}_n]}$.
	Since $\bar{X}_n\geq 0$, it follows from \eqref{eq:strategy.for.mathcalEn} and
	Lemma \ref{lem:integral.over.time.geq0} that 
	$\bar{\sigma}_n=T$ on $ \bar{\Delta}\cap\bar{\Xi}_n$. 
	The assumptions that $\bar{\Xi}=\bigcup_n\bar{\Xi}_n$ and $\bar{\Xi}_n\subset\bar{\Xi}_{n+1}$
	therefore imply 
	\begin{align*} 
	&m+\varepsilon + (\tilde{G}^n\cdot\bar{S})_t+ (\tilde{H}^n\cdot\bar{\mathbb{S}})_t\geq 0 
	\quad\text{on }\bar{\Omega} \ \mbox{ for all $t$ and $n$,}\\
	&m+\varepsilon + \liminf_n\Big( (\tilde{G}^n\cdot\bar{S})_T+ (\tilde{H}^n\cdot\bar{\mathbb{S}})_T\Big) \geq
	\liminf_n \bar{X}_n=\bar{X} \quad\text{on } \bar{\Delta}\cap\bar{\Xi},
	\end{align*}
	which shows that $\bar{\mathcal{E}}(\bar{X})\leq m+\varepsilon$. As $\varepsilon$ was arbitrary,
	the desired inequality follows.
	
	The proof of the theorem is now readily completed using the transfer principle stated in Lemma \ref{lem:transfer}.
\end{proof}

\subsection{Proof of Theorem~\ref{thm:integral.Xi.has.assumptions}}
\label{subsec:proof-integral-ass}

Throughout this subsection we consider the prediction set $\Xi\equiv\Xi_c$ defined in \eqref{Xi-c},
i.e.~every $\omega\in\Xi_c$ is H\"older continuous with $d\langle \omega\rangle/dt\leq c$. 

\begin{proposition}
\label{lem:BDG.Xi.has.assumptions}
	Let $F\in \mathcal{H}_{s,c}(H,K)$. 
	Then,  one has
	\[ \mathcal{E}\Big( \sup_{t\in [0,T]} \|(F \cdot S)_t\|_K^2 \Big)
	\leq 4 c \int_0^T  \mathcal{E}\big( \|F_t\|_{L(H,K)}^2\big) dt.\]
\end{proposition}

\begin{proof}
	First, since $\Omega\ni \omega \mapsto \sup_{t\in[0,T]} \|(F \cdot S)_t(\omega)\|_K^2$ is lower semicontinuous, the function
	\[C([0,T],H)\ni \omega \mapsto
	\sup_{\delta>0}\inf_{\substack{\tilde \omega\in\Omega\\ \|\tilde\omega-\omega\|_H\le\delta}} 
  \sup_{t\in[0,T]} \|(F \cdot S)_t(\tilde\omega)\|_K^2 \]
	defines its lower semicontinuous extension on $C([0,T],H)$. Moreover, since by assumption $\Xi\equiv\Xi_c$,
	  the conditions of Theorem \ref{thm:dual} are satisfied.
	Therefore, we can apply Theorem \ref{thm:dual} to obtain that
	\[ \mathcal{E}\Big( \sup_{t\in[0,T]} \|(F \cdot S)_t\|_K^2 \Big)
	=\sup_{P\in\mathcal{M}(\Xi)} E_P\Big[ \sup_{t\in[0,T]} \|(F \cdot S)_t\|_K^2 \Big]. \]
	Now, under every $P\in\mathcal{M}(\Xi)$, 
		the process $\|(F \cdot S)\|_K$ is a (real-valued) submartingale.
	Therefore, Doob's maximal inequality implies
	\[ E_P\Big[ \sup_{t\in[0,T]} \|(F \cdot S)_t\|_K^2 \Big]  
	\leq 4 E_P[  \|(F \cdot S)_T\|_K^2]
	\leq 4 E_P[ (\|F\|_{L(H,K)}^2 \cdot \langle S\rangle )_T]\]
	where the last inequality is the weak It\^o-Isometry 
	(apply e.g.~the same arguments as in the proof of Lemma \ref{lem:integrals.squared}
	and integrate w.r.t.~$P$). 
	Finally, by assumption and Lemma \ref{lem:weak.duality} one has
	\[E_P[ (\|F\|_{L(H,K)}^2 \cdot \langle S\rangle )_T]
	\leq c \int_0^T E_P[\|F_t\|_{L(H,K)}^2]\, dt
	\leq c \int_0^T \mathcal{E}(\|F_t\|_{L(H,K)}^2)\, dt \]
	which proves the claim.
\end{proof}

\begin{proof}[\textbf{Proof of Theorem \ref{thm:integral.Xi.has.assumptions}}]
	The part of Theorem \ref{thm:integral} which states that the stochastic integral exists for integrands in $\mathcal{H}^2(H,K)$ follows from
	Proposition \ref{lem:BDG.Xi.has.assumptions} using the exact same arguments as in the proof of Theorem \ref{thm:integral}.
	The second part is shown in Lemma~\ref{le:f-Lip-nice} below.
\end{proof}
{
\begin{proof}[\textbf{Proof of Remark~\ref{rem:integrand-2}}]
 Define $\tau_n:=(n/N)\wedge T$ for every $n\geq 0$ and then set $G:=\sum_{n=0}^\infty F_{\tau_n}1_{(\tau_n,\tau_{n+1}]}\in\mathcal{H}_{s,c}(H,K)$.
As both $F_t$ and $G_t$ are continuous for every $t\in[0,T]$, it follow from Theorem \ref{thm:dual} that 
\begin{align*}
\mathcal{E}(\|F_t-G_t\|_{L(H,K)}^2)
&= \sup_{P\in\mathcal{M}(\Xi)} E_P[\|F_t-G_t\|_{L(H,K)}^2]\\
&\leq \rho(1/N)^2 \Big( 1+ \sup_{P\in\mathcal{M}(\Xi)} E_P\Big[\sup_{r\in[0,T]}\langle S_r\rangle_T^{p/2}\Big]\Big)\\
&\leq  \rho(1/N)^2 (1+(Tc)^{p/2}).
\end{align*}
Indeed, the first inequality follows from the modulus of continuity assumption on $F$ together with the classical BDG-inequality applied under each $P$, and the second inequality by assumption that $d\langle S\rangle_t\leq c\,dt$ on $\Xi_c$.
Integrating over $t$ and taking the limit for $N\to\infty$ yields the claim.
\end{proof}
}
In the next Lemma, we extend the inequality obtained in Proposition~\ref{lem:BDG.Xi.has.assumptions} to integrands lying in $\mathcal{H}^2(H,K)$. Moreover, we prove for Lipschitz continuous functions $f\colon [0,T]\times K \mapsto L(H,K)$, that $f(\cdot,(F\cdot S))\in\mathcal{H}^2(H,K)$ whenever $F \in \mathcal{H}^2(H,K)$. This is the crucial property allowing to solve stochastic differential equations for typical paths. 

\begin{lemma}
	\label{le:f-Lip-nice}
Let $F\in\mathcal{H}^2(H,K)$. Then, one has
\[ \mathcal{E}\Big( \sup_{t\in[0,T]} \|(F \cdot S)_t\|_K^2 \Big)
	\leq 4c \int_0^T  \mathcal{E}\big( \|F_t\|_{L(H,K)}^2\big)\, dt.\]
In addition, if $f\colon [0,T]\times K \mapsto L(H,K)$ is Lipschitz continuous 
then the map
$\Omega \times [0,T] \ni (\omega,t) \mapsto f(t,(F\cdot S)_t(\omega)) \in L(H,K)$ is an element of $\mathcal{H}^2(H,K)$.
\end{lemma}

\begin{proof}
	The first part follows from Proposition \ref{lem:BDG.Xi.has.assumptions} and the triangle inequality.

	It remains to prove that $f(\cdot,(F\cdot S))\in \mathcal{H}^2(H,K)$. Assume first that $F\in\mathcal{H}_{s,c}(H,K)$ and define 
	\[H^n:=\sum_{i=0}^{n-1} f(iT/n,(F\cdot S)_{iT/n}) 1_{(iT/n,(i+1)T/n]}\]
	for every $n$.
	Since $\omega\mapsto (F\cdot S)_{iT/n}(\omega)$ is continuous for ever $i$,
	one has $H^n\in\mathcal{H}_{s,c}(H,K)$.
	Moreover, with $\pi_n(t):=\max\{ iT/n :i\in\mathbb{N}\mbox{ such that }iT/n\le t\}$ and 
    $L_f$ being the Lipschitz constant of $f$, one has
	\begin{align}
	\|f(\cdot,(F\cdot S))-H^n\|_{\mathcal{H}^2(H,K)}^2
	&\leq \int_0^T2L_f^2 \mathcal{E}\Big( |t-\pi_n(t)|^2+ \|(F\cdot S)_t-(F\cdot S)_{\pi_n(t)}\|_K^2\Big) \,dt \nonumber\\
	&\leq 2L_f^2\int_0^T |t-\pi_n(t)|^2 + \mathcal{E}\big(\| (F\cdot S)_t-(F\cdot S)_{\pi_n(t)}\|_K^2\big) \,dt. \label{eq:rem:explanation1}
	\end{align}
	Now Theorem \ref{thm:dual}, the weak It\^o-Isometry (argue as in Proposition \ref{lem:BDG.Xi.has.assumptions}), 
	and weak duality (Lemma \ref{lem:weak.duality}) imply for every $t$ that
	\begin{align}
	&\mathcal{E}\big(\|(F\cdot S)_t-(F\cdot S)_{\pi_n(t)}\|_K^2\big)
	=\sup_{P\in\mathcal{M}(\Xi)}E_P\big[\|(F\cdot S)_t-(F\cdot S)_{\pi_n(t)}\|_K^2\big] \nonumber\\
	&\leq \sup_{P\in\mathcal{M}(\Xi)}E_P\Big[\int_{\pi_n(t)}^t\|F_s\|_{L(H,K)}^2\, d\langle S\rangle_s \Big]
	\leq c \int_{\pi_n(t)}^t \mathcal{E}\big( \|F_s\|_{L(H,K)}^2\big)\, ds \label{eq:rem:explanation2}\\
	&\leq  c \int_{0}^T \mathcal{E}\big( \|F_s\|_{L(H,K)}^2\big)\, ds <+\infty.	 \nonumber
	\end{align}
This shows that $\mathcal{E}\big(\|(F\cdot S)_t-(F\cdot S)_{\pi_n(t)}\|_K^2\big)$ is dominated by $c \int_{0}^T \mathcal{E}\big( \|F_s\|_{L(H,K)}^2\big)\, ds <+\infty$ and 
converges pointwise to $0$ when $n$ goes to infinity since then $\pi_n(t)\to t$. Therefore, dominated convergence implies $\|f(\cdot,(F\cdot S))-H^n\|_{\mathcal{H}^2(H,K)}\to 0$, which shows that $f(\cdot,(F\cdot S))\in\mathcal{H}^2(H,K)$.

The general case follows by approximating $F\in\mathcal{H}^2(H,K)$ by $F^n\in\mathcal{H}_{s,c}(H,K)$, and using the inequality
\begin{align*}
\|f(\cdot,(F^n\cdot S))-f(\cdot,(F\cdot S))\|^2_{\mathcal{H}^2(H,K)}&\le L^2_f \int_0^T \mathcal{E}\!\big(\|(F^n\cdot S)_t-(F\cdot S)_t\|^2_K\big)\,dt \nonumber\\
&\le T L^2_f \mathcal{E}\Big(\sup_{0\le t\le T}\|(F^n\cdot S)_t-(F\cdot S)_t\|^2_K\Big) \\
&\le  4c T L^2_f \|F^n-F\|^2_{\mathcal{H}^2(H,K)}, \nonumber
\end{align*}
where the last inequality is ensured by the first part. 
\end{proof}

\begin{remark}\label{rem:identification}
Let $F\in\mathcal{H}^2(H,K)$. Similar to the proof of Lemma \ref{le:f-Lip-nice} one can verify 
\begin{enumerate}
\item $f(\cdot, F)\in\mathcal{H}^2(H,K)$ for every function $f:[0,T]\times K\to L(H,K)$ which is Lipschitz continuous, and
\item $(F\cdot S)$ can be identified with an element in $\mathcal{H}^2(K,\mathbb{R})$, by considering $i(F\cdot S)$ for the isometric isomorphism $i:K\to L(K,\mathbb{R})$ given by the Riesz representation theorem.
\end{enumerate}
In Subsection \ref{sec:Picard} we will frequently use this identification.
\end{remark}


\subsection{Proof of Theorem \ref{thm:sde}}\label{sec:Picard}
We recall that throughout this subsection we consider the particular prediction set $\Xi_c$ defined in \eqref{Xi-c} and that  Assumption
~\ref{ass:SDE} hold.
\begin{lemma}
\label{lem:intA}
	For $F\in\mathcal{H}^2(\mathbb{R},K)$, the integral 
	\[(F\cdot A)\colon\Omega\to C([0,T],K)\]
	exists and satisfies
	\[ \mathcal{E}\Big(\sup_{t\in[0,T]} \|(F\cdot A)_t\|_K^2 \Big)
	\leq c^2T \int_0^T \mathcal{E}(\|F_t\|^2_{L(\mathbb{R},K)})\, dt.\]
	In particular, $(F\cdot A)\in\mathcal{H}^2(K,\mathbb{R})$ by identifying $K$ with $L(K,\mathbb{R})$.
	Moreover, if $f\colon [0,T]\times K\to L(\mathbb{R},K)$ is
	Lipschitz continuous, then $f(\cdot,(F\cdot A))\in\mathcal{H}^2(\mathbb{R},K)$. 
\end{lemma}
\begin{proof}
	For $F,G\in\mathcal{H}_{s,c}(\mathbb{R},K)$, Theorem \ref{thm:dual}
	and H\"older's inequality implies
	\begin{align*} 
	&\mathcal{E}\Big(\sup_{t\in[0,T]} \|(F\cdot A)_t-(G\cdot A)_t\|_K^2 \Big)
	\leq \sup_{P\in\mathcal{M}(\Xi)} E_P\Big[ \Big(\int_0^T \|F_t-G_t\|_K c\,dt \Big)^2 \Big]\\
	&\leq c^2T \int_0^T \sup_{P\in\mathcal{M}(\Xi)}E_P[\|F_t-G_t\|_K^2] \, dt
	\leq c^2T \|F-G\|_{\mathcal{H}^2(\mathbb{R},K)}^2.
	\end{align*}
	The rest follows by the same arguments as in the proof of Theorem \ref{thm:integral}.

	The second part can be proved as in Lemma \ref{le:f-Lip-nice}.
\end{proof}

In line with Remark \ref{rem:identification} the following holds.

\begin{remark}
\label{rem2:identification}
  For $F\in\mathcal{H}^2(K,\mathbb{R})$ one has 
  \begin{enumerate}
  \item $f(\cdot, F)\in\mathcal{H}^2(\mathbb{R},K)$ for every function $f:[0,T]\times K\to L(\mathbb{R},K)$ which is Lipschitz continuous, and
  \item $(F\cdot A)$ can be identified with an element in $\mathcal{H}^2(K,\mathbb{R})$.
  \end{enumerate}
\end{remark}

Now we are ready to prove Theorem \ref{thm:sde}. Consider the Picard iteration 
\[X^{n+1}:=x_0+(\mu(\cdot,X^n) \cdot A)+ (\sigma(\cdot,X^n)\cdot S) \]
starting at $X^0\equiv x_0\in K$ and recall the Lipschitz continuity of $\mu$ and $\sigma$ imposed in \eqref{eq:Lipschitz-coeff} with corresponding Lipschitz constant $L$ .

\begin{lemma}
	\label{lem:aprori}
	Assume that $X^n:\Omega \to C([0,T],K)$ and $X^n\in \mathcal{H}^2(K,\mathbb{R})$.
  Then the process $X^{n+1}:\Omega\to C([0,T],K)$ satisfies $X^{n+1}\in\mathcal{H}^2(H,K)$ and 
	\[g^{n+1}(t):=\mathcal{E}\Big(\sup_{s\in[0,t]} \|X_s^{n+1}-X_s^n\|_K^2\Big)
	\leq C \int_0^t g^n(s)\,ds\]
	for all $t\in [0,T]$ where 
  $g^0\equiv \sup_{t\in[0,T]}\|\mu(t,x_0)\|_{L(\mathbb{R},K)}  
  +\sup_{t\in[0,T]}\|\sigma(t,x_0)\|_{L(H,K)}<+\infty$
  and
  $C:=(2c^2 T+8c)L^2$.
\end{lemma}
\begin{proof}
	By Lemma \ref{le:f-Lip-nice}, Remark \ref{rem:identification}, Lemma \ref{lem:intA}, and Remark \ref{rem2:identification} it holds that $X^{n+1}:\Omega\to C([0,T],K)$ and $X^{n+1}\in\mathcal{H}^2(H,K)$. Define $\Delta\mu^n_t:=\mu(t,X^n_t)-\mu(t,X^{n-1}_t)$,
	$\Delta\sigma^n_t:=\sigma(t,X^n_t)-\sigma(t,X^{n-1}_t)$, and
	$\Delta X^n_t:=X^n_t-X^{n-1}_t$ for all $t\in[0,T]$ and $n\in\mathbb{N}$.
	Since
	\[\sup_{s\in[0,t]} \|\Delta X_s^{n+1}\|_K^2
	\leq 2 \sup_{s\in[0,t]} \|(\Delta\mu^n\cdot A)_s\|_K^2
		+ 2\sup_{s\in[0,t]}\|(\Delta\sigma^n \cdot S)_s\|_K^2,\]
	it follows from Lemma \ref{le:f-Lip-nice} and Lemma \ref{lem:intA} that
	\begin{align*}
	g^{n+1}(t)
	&=\mathcal{E}\Big(\sup_{s\in[0,t]} \|\Delta X_s^{n+1}\|_K^2\Big) \\
	&\leq 2 \mathcal{E}\Big(\sup_{s\in[0,t]} \|(\Delta \mu^n\cdot A)_s \|_K^2\Big) 
		+ 2 \mathcal{E}\Big(\sup_{s\in[0,t]} \|(\Delta \sigma^n\cdot S)_s\|_K^2\Big) \\
	&\leq 2c^2 T\int_0^t \mathcal{E}\Big(\|\Delta\mu^n_s\|^2_{L(\mathbb{R},K)} \Big)\,ds 
		+ 8c \int_0^t \mathcal{E}\Big(\|\Delta\sigma^n_s\|^2_{L(H,K)} \Big)\,ds \\
	&\leq C \int_0^t \mathcal{E}\Big(\|\Delta X^n_s\|^2_{L(K,\mathbb{R})} \Big) \,ds \\
	&\leq C \int_0^t \mathcal{E} \Big(\sup_{u\in[0,s]} \|\Delta X_u^{n}\|_K^2\Big) \,ds 
	= C \int_0^t g^n(s) \,ds.
	\end{align*}
	For $n=0$ the previous computation yields $g^1(t)\leq C\int_0^t g^0(s)\, ds$.
  That $g^0<+\infty$ follows directly from the assumption \eqref{eq:Lipschitz-coeff} on $\mu$ and $\sigma$.
	\end{proof}

\begin{proof}[\textbf{Proof of Theorem \ref{thm:sde}}]
The existence and uniqueness of the SDE follows directly from Lemma \ref{lem:aprori} by standard arguments. Indeed, iterating the estimate $g^{n+1}(t)\le C\int_0^t g^n(s)\, ds$ yields
\[g^n(t)\le g^0(0) \frac{(Ct)^n}{n !}\]
for all $t\in[0,T]$ and $n\in\mathbb{N}$.
For any function $Y\colon\Omega\to C([0,T],K)$ define $\|Y\|:=\mathcal{E}\big(\sup_{t\in[0,T]}\|Y_t\|_K^2\big)^{\frac{1}{2}}$
which by Proposition \ref{prop:L2-Banach-Banach} is a semi-norm. Then, since for every $m,n\in\mathbb{N}$ with $m>n$ one has
\[\|X^m-X^n\|
\leq \sum_{k>n}\| X^k-X^{k-1}\|
\leq \sum_{k>n}\Big(\frac{(CT)^k}{k!}\Big)^{\frac{1}{2}}
\to 0\quad\mbox{as }n\to\infty\]
it follows that $(X^n)$ is a Cauchy sequence w.r.t.~$\|\cdot\|$. By Proposition \ref{prop:L2-Banach-Banach}
there exists a subsequence $(n_k)$ such that $X^{n_k}$ converges to some $X:\Omega\to C([0,T],K)$ for typical paths and $\|X^n-X\|\to 0$. Since
\[\|X^n-X\|^2_{\mathcal{H}^2(K,\mathbb{R})}=\int_0^T \mathcal{E}\big(\|X^n_s-X_s\|^2_{L(K,\mathbb{R})}\big) \,ds
\le T\|X^n-X\|^2\]
it follows that $X\in \mathcal{H}^2(K,\mathbb{R})$, as well as
\[\|\mu(\cdot,X^n)-\mu(\cdot,X)\|_{\mathcal{H}^2(\mathbb{R},K)}\to 0\quad \mbox{and}\quad \|\sigma(\cdot,X^n)-\sigma(\cdot,X)\|_{\mathcal{H}^2(H,K)}\to 0\]
by Lipschitz continuity of $\mu$ and $\sigma$. This shows 
\begin{align*}
X&=\lim_{n\to\infty} X^{n+1} 
=\lim_{n\to\infty}\Big(x_0+(\mu(\cdot,X^n)\cdot A)+(\sigma(\cdot,X^n)\cdot S)\Big)\\
&= x_0+(\mu(\cdot,X)\cdot A)+(\sigma(\cdot,X)\cdot S)
\end{align*}
where the limits are w.r.t.~$\|\cdot\|$.

As for the uniqueness, suppose there exist two solutions $X,Y:\Omega\to C([0,T],K)$ in $\mathcal{H}^2(K,\mathbb{R})$ of the SDE. Similar to the proof of Lemma \ref{lem:aprori} we get
\[g(t):=\mathcal{E}\Big(\sup_{s\in[0,t]} \|X_s-Y_s\|^2_K\Big)\le C\int_0^t g(s)\, ds\]
with $C:=(2c^2 T+8c)L^2$. Iterating this estimate yields $g(t)\le g(T)\frac{(Ct)^k}{k!}$ for all $k\in\mathbb{N}$ and $t\in[0,T]$, so that $\|X-Y\|=g(T)^{\frac{1}{2}}=0$. This shows that $X$ and $Y$ coincide for typical paths. 
\end{proof}
{\subsection*{Acknowledgments}
We would like to thank the anonymous referee for a very thorough reading and exceptionally helpful comments.\\
The first author is supported by the Austrian Science Fund (FWF) under grant P28861.\\
 The third author is supported by the NAP Grant.\\
Part of this research was carried out while the authors were visiting the Shanghai Advanced 
Institute of Finance and the School of Mathematical Sciences at the Shanghai Jiao Tong University 
in China, and we would like to thank Samuel Drapeau for his hospitality.}
\newpage
\bibliographystyle{abbrv}

\end{document}